\documentclass{amsart}
\usepackage{amscd}
\usepackage{amsfonts}
\usepackage{amsmath}
\usepackage{amssymb}
\usepackage{amsthm}
\usepackage{fancyhdr}
\usepackage{latexsym}

\newtheorem{theorem}{Theorem}
\newtheorem{lemma}[theorem]{Lemma}
\newtheorem{remark}[theorem]{Remark}

\newcommand{\R}{\mathbb R}

\newcommand{\p}{\partial}

\begin{document}

\title[Ill-posedness for the $b$-Novikov equation]{Ill-posedness in the critical Sobolev space for the $b$-Novikov equation}

\author{Dan-Andrei Geba, A. Alexandrou Himonas and Curtis Holliman}

\address{Department of Mathematics, University of Rochester, Rochester, NY 14627, U.S.A.}
\email{dangeba@math.rochester.edu}
\address{Department of Mathematics, University of Notre Dame, Notre Dame, IN 46556, U.S.A.}
\email{himonas.1@nd.edu}
\address{Department of Mathematics, The Catholic University of America, Washington, DC 20064, U.S.A.}
\email{holliman@cua.edu}
\date{}

\begin{abstract}
This article proves norm inflation in the critical Sobolev space  $H^{3/2}(\R)$ for  the $b$-Novikov equation, which is a $1$-parameter family of Camassa-Holm-type equations with cubic nonlinearities. This result completes the well-posedness theory for this equation, which was previously known to be locally well-posed in $H^{s}(\R)$ for $s>3/2$ and ill-posed in  $H^{s}(\R)$ for $s<3/2$.
\end{abstract}


\maketitle


\section{Introduction}

In this paper, we consider the Cauchy problem for the $b$-Novikov equation ($b$NE) on the line, which is given by
\begin{equation}
\begin{cases}
(1-\p^2_x)u_t= u^2u_{xxx}+bu u_x u_{xx}- (b+1)u^2u_x, &\quad u=u(x,t),\\
u(x,0)=u_0(x), &\quad x\in \R, \quad t\in \R,
\end{cases}
\label{cp}
\end{equation}
and we show that it is ill-posed (IP) in the Sobolev space $H^{3/2}(\R)$. The $b$NE can be seen as part of the family of generalized Camassa-Holm equations (also known as the g-$kb$CH equation)
\begin{equation}
(1-\p^2_x)u_t= u^ku_{xxx}+bu^{k-1} u_x u_{xx}- (b+1)u^ku_x,\label{gkb}
\end{equation}
where $k\in \mathbb{Z}^+$ and $b\in \R$, which also includes the $b$-family of equations
\begin{equation}
(1-\p^2_x)u_t= u^ku_{xxx}+b u_x u_{xx}- (b+1)u u_x \label{b}.
\end{equation}

Notable members of \eqref{gkb} are the Camassa-Holm equation (CH) (corresponding to $(k,b)=(1,2)$), the Degasperis-Procesi equation (DP) (corresponding to $(k,b)=(1,3)$), and the Novikov equation (NE) (corresponding to $(k,b)=(2,3)$), which are the only integrable g-$kb$CH equations. By an integrable equation, we understand one that admits infinitely many conserved quantities, a Lax pair, a bi-Hamiltonian formulation, and which can be solved by the inverse scattering method. A distinct feature of \eqref{gkb}  is that all equations possess \textbf{peakon travelling wave solutions} both on the line and on the circle. On the line, the peakon formula is given by
\[
u_c(t,x)=c^{1/k}e^{-|x-ct|},
\]   
where $c$ is any positive constant. 

Both CH and DP have a solid physical foundation, as they arise as shallow water wave models describing water wave propagation in the shallow regime (see, e.g., Lannes \cite{L13}, Whitham \cite{W74}). The derivation of CH by Camassa and Holm \cite{CH93} was partly motivated by the question of whether water waves can simultaneously exhibit wave breaking and peaked traveling wave solutions. It is also noteworthy that CH had been discovered earlier by Fokas and Fuchssteiner \cite{FF81} in the context of hereditary symmetries. More broadly, the $b$-family of equations - which includes both CH and DP - was introduced by Holm and Staley \cite{HS03,HS032} as a one-dimensional model for active fluid transport.

By comparison, the discovery of NE arose from an integrability study conducted by Novikov \cite{N09} within the broader class of Camassa-Holm type equations
\[
(1-\p^2_x)u_t=P(u, u_x, u_{xx}, \ldots),
\]
where $P$ denotes a polynomial in $u$ and its spatial derivatives. This investigation led to the identification of more than twenty integrable equations with quadratic nonlinearities - including CH and DP - as well as more than ten integrable equations with cubic nonlinearities, among which NE is a prominent example.

All of the equations described above are nonlinear evolution equations, and their well-posedness (WP) theory in the sense of Hadamard (i.e., existence, uniqueness, and continuity of the data-to-solution map) therefore constitutes a natural and fundamental object of study. This subject has received considerable attention in the literature, with contributions by Rodriguez-Blanco \cite{R01} and Li and Olver \cite{LO00} for CH, Gui and Liu \cite{GL11} for DP, Gui, Liu, and Tian \cite{GLT08} for the $b$-family of equations, and Himonas and Holliman \cite{HH12} for NE, among others. Subsequently, Himonas and Holliman \cite{HH14} established in a unified framework that the g-$kb$CH equation is locally WP in $H^s$ for $s>3/2$, both on the real line and on the circle. Moreover, they showed that, within the same Sobolev range, the associated data-to-solution map fails to be uniformly continuous between the corresponding functional spaces. This lack of uniform continuity is consistent with the quasilinear nature of the g-$kb$CH equation.

Complementing the WP theory, there is by now a fairly comprehensive body of results concerning IP for these equations. Notable contributions include Byers \cite{B06} and Himonas and Kenig \cite{HK09}  for CH, Himonas, Holliman, and Grayshan \cite{HHG14} for DP, Himonas, Grayshan, and Holliman \cite{HGH16} for the $b$-family of equations, Himonas, Holliman, and Kenig \cite{HHK18} for the Novikov equation, and Himonas and Holliman \cite{HH22} for the $b$NE. Collectively, these works establish IP in $H^s$ for $s<3/2$, both on the real line and on the circle, by demonstrating either nonuniqueness of solutions or norm inflation, the latter implying discontinuity of the data-to-solution map. For instance, Himonas, Holliman, and Kenig \cite{HHK18} showed that the Novikov equation exhibits nonuniqueness when $s\leq 5/4$ and norm inflation in the range $5/4<s<3/2$. 

In light of the WP and IP results described above, it is natural to regard $s=3/2$ as the critical Sobolev index for these equations and to ask whether this endpoint belongs to the WP or IP regime. In a notable result - particularly striking for its simplicity - Guo, Liu, Molinet, and Yin \cite{GLM19} resolved this question by showing that the $b$-family of equations is IP in $H^{3/2}$, both on the real line and on the circle. The same work also established IP of NE in $H^{3/2}$ on the real line. More precisely, building crucially on previously known blow-up results for these equations, Guo et al. \cite{GLM19} proved that norm inflation occurs in any of the Besov space $B^{1+1/p}_{p,r}$  with $1\leq p\leq \infty$ and $1< r\leq \infty$; note in particular that $B^{3/2}_{2,2}=H^{3/2}$.

This result naturally led us to investigate the behavior of the $b$NE in $H^{3/2}(\R)$. In this direction, our main result can be stated as follows:

\begin{theorem}
The $b$NE is IP in $H^{3/2}(\R)$ when $b\geq 3$. More precisely, we have norm inflation in the sense that, for every $\epsilon>0$, there exist a solution $u_\epsilon$ to the $b$NE and $0<t_\epsilon<\epsilon$ such that
\begin{equation}
\|u_\epsilon(0)\|_{H^{3/2}(\R)}\lesssim \epsilon \label{0}
\end{equation}
and
\begin{equation}
\|u_\epsilon(t_\epsilon)\|_{H^{3/2}(\R)}\gtrsim \epsilon^{-1}.\label{-1}
\end{equation}
\label{main}
\end{theorem} 

Throughout, we rely on the classical notation $A\lesssim B$ to denote the inequality $A\leq c\,B$, where $c>0$ is a constant. Correspondingly, $A\sim B$ signifies that both $A\lesssim B$ and $B\lesssim A$ hold. Moreover, the notation $A\ll B$ indicates that $A\leq c\,B$ holds with a constant $c>0$ that can be chosen to be sufficiently small.
 
A few remarks regarding the proof of this theorem are in order. A fundamental obstacle is the absence of any known blow-up results for the $b$NE when $b\neq 3$. This appears to be intrinsically linked to the fact that NE is the sole integrable member of the $b$NE family. As a result, the approach introduced by Guo et al. \cite{GLM19} is not applicable in our setting. In particular, our argument proceeds along a different route and, as a byproduct, provides an independent proof of IP in $H^{3/2}(\R)$ for NE.

The paper is organized as follows. In Section 2, we revisit the WP result of Himonas and Holliman \cite{HH14} for the g-$kb$CH equation and derive a sharper lower bound for the lifespan of solutions to the Cauchy problem \eqref{cp}. Section 3 is devoted to the proof of a key differential inequality for a nonlinear quantity evaluated along particle trajectories associated with the $b$NE.. In the final section, we first reformulate our main result and construct suitable initial data leading to norm inflation, building on a sequence originally introduced by Guo et al. \cite{GLM19} in their study of NE. We then combine the analytical tools developed in the preceding sections to complete the proof of our main result by contradiction.


\section{Relevant WP theory for the $b$NE}
As mentioned in the introduction, Himonas and Holliman proved in \cite{HH14} a WP result (stated there as Theorem 1) for the  g-$kb$CH equation \eqref{gkb}. We restate it below in the form relevant to the $b$NE Cauchy problem \eqref{cp}.

\begin{theorem}
Let $s>3/2$. If $u_0\in H^s(\R)$, then there exists $T=T( \|u_0\|_{H^s})>0$ and a unique solution $u\in C([0, T]; H^s)$ to the Cauchy problem \eqref{cp}. Moreover, the solution depends continuously on the initial data and satisfies the estimate 
\[
\|u(t)\|_{H^s}\leq 2 \|u_0\|_{H^s}, \qquad \forall\, 0\leq t\leq T\leq \frac{1}{c_s\|u_0\|^2_{H^s}},
\]
where the constant $c_s>0$ depends only on $s$.
\label{hh}
\end{theorem}

In our analysis, we will apply this theorem for values of $s$ close to $3/2$. Because the argument is delicate near this endpoint, it is essential to determine the precise behavior of the constant $c_s$ as $s\to 3/2+$. We claim that one may take 
\begin{equation}
c_s\sim \frac{1}{2s-3}.
\label{cs}
\end{equation}

To justify this, we start with the classical Sobolev embedding $H^\sigma(\R)\subset L^\infty(\R)$ when $\sigma>1/2$. A standard proof proceeds via
\begin{equation*}
\aligned
\|f\|_{L^\infty(\R)}\lesssim \|\hat{f}\|_{L^1(\R)}&\lesssim \|(1+\xi^2)^{-\sigma/2}\|_{L^2(\R)} \|(1+\xi^2)^{\sigma/2}\hat{f}\|_{L^2(\R)}\\
&\sim \left(\int_\R(1+\xi^2)^{-\sigma}\,d\xi\right)^{1/2}\|f\|_{H^\sigma(\R)},
\endaligned
\end{equation*}
where
\[
\hat{f}(\xi)=\int_\R e^{-ix\xi}\,f(x)\,dx
\] 
denotes the Fourier transform. Using the $\Gamma$ function and its properties, one obtains
\begin{equation*}
\int_\R(1+\xi^2)^{-\sigma}\,d\xi= \frac{\pi^{1/2}\,\Gamma(\sigma-1/2)}{\Gamma(\sigma)}\sim \frac{1}{\sigma-1/2}
\end{equation*}
as $\sigma\to 1/2+$. Consequently,
\begin{equation}
\|f\|_{L^\infty(\R)}\lesssim \frac{1}{(\sigma-1/2)^{1/2}}\|f\|_{H^\sigma(\R)}
\label{inf-s}
\end{equation}
as $\sigma\to 1/2+$.

We next recall the standard product estimate
\begin{equation*}
\|f\,g\|_{H^\sigma(\R)}\leq c_\sigma\left(\|f\|_{L^\infty(\R)}\|g\|_{H^\sigma(\R)}+\|f\|_{H^\sigma(\R)}\|g\|_{L^\infty(\R)}\right),
\end{equation*}
valid for $\sigma>0$, with $c_\sigma>0$ being a constant depending on $\sigma$. Since $1/2$ lies strictly within the admissible range for $\sigma$, we infer based on \eqref{inf-s} that
\begin{equation*}
\|f g\|_{H^\sigma(\R)}\lesssim \frac{1}{(\sigma-1/2)^{1/2}}\|f\|_{H^\sigma(\R)}\|g\|_{H^\sigma(\R)},
\end{equation*}
and, hence,
\begin{equation*}
\|f g h\|_{H^\sigma(\R)}\lesssim \frac{1}{\sigma-1/2}\|f\|_{H^\sigma(\R)}\|g\|_{H^\sigma(\R)}\|h\|_{H^\sigma(\R)},
\end{equation*}
as $\sigma\to 1/2+$.

Using this trilinear bound, we revisit Lemma 1 of \cite{HH14} in the setting of the $b$NE. The corresponding estimate becomes
\begin{equation}
\|F(u)\|_{H^s(\R)}\lesssim \frac{1}{2s-3}\|u\|^3_{H^s(\R)}\label{fu}
\end{equation}
as $s\to 3/2+$, where
\begin{equation*}
F(u)= \frac{b}{3}\p_x D^{-2}(u^3)+\frac{6-b}{2}\p_x D^{-2}(u u_x^2)+\frac{b-2}{2}D^{-2}(u_x^3)
\end{equation*}
and 
\[
\widehat{D^p f}(\xi)= (1+\xi^2)^{p/2}\hat{f}(\xi).
\]
Note that the $b$NE can equivalently be rewritten as
\begin{equation*}
u_t+u^2u_x+F(u)=0.
\end{equation*}

Besides \eqref{fu}, the remaining key ingredient in the derivation of the lifespan $T$ in Theorem \ref{hh} is the Kato-Ponce commutator estimate 
\begin{equation*}
\|[D^\sigma, f]\,g\|_{L^2(\R)}\leq c_\sigma\left(\|\p_xf\|_{L^\infty(\R)}\|D^{\sigma-1}g\|_{L^2(\R)}+\|D^\sigma f\|_{L^2(\R)}\|g\|_{L^\infty(\R)}\right),
\end{equation*}
valid for $\sigma>0$, with $c_\sigma>0$ being a constant depending on $\sigma$. Since $3/2$ lies safely within the admissible range, we can combine the previous estimates to obtain the following refinement of estimate (2.30) in \cite{HH14} for the $b$NE:
\begin{equation*}
\frac{1}{2}\frac{d}{dt}\|u_\epsilon(t)\|^2_{H^s(\R)}\lesssim \frac{1}{2s-3} \|u_\epsilon(t)\|^4_{H^s(\R)}
\end{equation*}
as $s\to 3/2+$, where $u_\epsilon=u_\epsilon(x,t)$ solves the mollified Cauchy problem
\begin{equation*}
\begin{cases}
v_t+J_\epsilon[(J_\epsilon v)^2 J_\epsilon v_x]+F(v)=0,\\
v(0,x)=u_0(x).
\end{cases}
\end{equation*}
Above, $(J_\epsilon)_{0<\epsilon\leq 1}$ are the standard Friedrichs mollifiers. Proceeding as in \cite{HH14}, one concludes  that the lifespan $T=T( \|u_0\|_{H^s})$ in Theorem \ref{hh} may be chosen so that
\begin{equation*}
T\leq \frac{c(2s-3)}{\|u_0\|^2_{H^s}},
\end{equation*}
where $c>0$ is an absolute constant. This establishes \eqref{cs}.


\section{Crucial differential inequality}\label{3}
In this section we establish a differential inequality that plays a central role in the norm inflation argument. Similar inequalities have appeared for other members of the g-$kb$CH family, where they were combined either with symmetry properties (as in the $b$CH equation) or with integrability (as in the NE) to deduce blow-up results. In contrast, for the $b$NE with $b\neq 3$, neither symmetry nor integrability is available. As a result, the analysis is substantially more delicate, since the differential inequality cannot be supplemented by additional structural properties to yield blow-up statements.


\subsection*{Preliminaries}
We anticipate and disclose that we will work with real-valued initial data $u_0\in H^\infty(\R)$ for \eqref{cp}, and we assume that
\[
y_0=(1-\partial^2_x)u_0
\]
satisfies
\begin{equation}\label{y0}
y_0\geq 0\ \text{on}\ (-\infty, 0], \quad y_0(0)=0, \quad y_0\leq 0 \ \text{on}\ [0,\infty).
\end{equation}
By Theorem \ref{hh}, there exists a unique (maximal) solution
\begin{equation}
u\in C([0,T_{\text{max}}); H^3(\R))\cap C^1([0,T_{\text{max}}); H^{2}(\R))\label{uh3}
\end{equation}
to \eqref{cp} corresponding to such initial data.

We next introduce the particle trajectory (flow) associated with $u$. For each $x\in \R$, let $q(x,t)$ solve
\begin{equation*} 
\begin{cases}
\frac{dq(x,t)}{dt}=u^2(q(x,t), t),\\
q(x,0)=x.
\end{cases}
\end{equation*}
In view of \eqref{uh3}, standard ODE theory ensures that this problem admits a unique solution 
\[
q\in C^1(\R\times [0,T_{\text{max}}); \R).
\] 
Differentiating the equation for $q$ with respect to $x$, we obtain
\begin{equation*}
q_x(x,t)=e^{2\int_0^t uu_x(q(x,s),s)\,ds}, \quad q_x(x,0)=1.
\end{equation*}
Hence, for each $t\in [0,T_{\text{max}})$, the map $x\to q(x,t)$ is an increasing diffeomorphism of $\R$. 

Recall that the $b$NE can equivalently be written as
\begin{equation}
y_t+u^2y_x+buu_x\,y=0, \qquad y=(1-\partial^2_x)u.\label{bne}
\end{equation}
Using this formulation, we compute
\begin{equation*}
\frac{d}{dt}\left\{y(q(x,t),t)q_x^{b/2}(x,t)\right\}=(y_t+u^2y_x+buu_x\,y)(q(x,t),t)\,q_x^{b/2}(x,t)=0,
\end{equation*}
and therefore
\begin{equation*}
y(q(x,t),t)q_x^{b/2}(x,t)=y(x,0)=y_0(x), \quad \forall\, (x,t)\in \R\times [0,T_{\text{max}}).
\end{equation*}
In view of \eqref{y0}, it follows that for each $t\in [0,T_{\text max})$, 
\begin{equation*}
y(q(x,t),t)\geq 0,\,(\forall)\,x< 0,\qquad y(q(0,t),t)= 0,\qquad y(q(x,t),t)\leq 0,\,(\forall)\,x> 0. 
\end{equation*}
Since $x\to q(x,t)$ is an increasing diffeomorphism of $\R$, we deduce
\begin{equation}
y(z,t)\geq 0,\,(\forall)\,z< q(0,t), \qquad y(z,t)\leq 0,\,(\forall)\,z> q(0,t). \label{yq0}
\end{equation}


\subsection*{Relevant integral identities}
Due to $y=(1-\partial^2_x)u$, we have the representation
\begin{equation*}
\aligned
u(x,t)=\int_\R \frac{1}{2} e^{-|x-\xi|}\,  y(\xi,t)\, d\xi=
\frac{1}{2} e^{-x} 
\int_{-\infty}^x  e^{ \xi}\, y(\xi,t) \,d\xi 
+
\frac{1}{2}
e^{x}
\int_x^\infty  e^{-\xi } \,y(\xi,t)\, d\xi.  
\endaligned
\end{equation*}
Differentiating with respect to $x$ yields
\begin{equation*}
u_x(x,t) = 
- \frac{1}{2} e^{-x} \int_{-\infty}^x e^\xi\, y(\xi,t)\, d\xi
+
\frac{1}{2} e^x 
\int_{x}^\infty  e^{ - \xi} \,y(\xi,t)\, d\xi. 
\end{equation*}
Consequently,
\begin{equation}\left\{\aligned
(u+u_x)(x,t) = e^{x} \int_x^{\infty} e^{-\xi}\, y(\xi,t)\, d\xi,\\
(u-u_x)(x,t) = e^{-x} \int_{-\infty}^x e^\xi\, y(\xi,t)\, d\xi.
\endaligned
\right.
\label{uux}
\end{equation}

From this point onward, we evaluate all quantities at
\begin{equation*}
(x,t)=(q(0,t), t),
\end{equation*}
using that
\begin{equation*}
q_t(0,t)=u^2(q(0,t), t)\qquad\text{and}\qquad y(q(0,t),t)= 0.
\end{equation*}
Employing the identity
\begin{equation*}
2 u u_x =  \frac{1}{2}(u + u_x)^2 -  \frac{1}{2}(u - u_x)^2,
\end{equation*}
a direct computation gives
\begin{equation}
\aligned
\frac{d}{dt}\left\{2 u u_x(q(0,t), t)\right\}=\,&2u^4(q(0,t), t)+2u^2u_x^2(q(0,t), t)\\&+ (u + u_x)(q(0,t), t) \cdot e^{q(0,t)} \int_{q(0,t)}^\infty e^{-\xi} \,y_t (\xi,t)\, d\xi\\&- (u - u_x)(q(0,t), t) \cdot e^{-q(0,t)} \int_{-\infty}^{q(0,t)} e^\xi\, y_t (\xi, t)\, d\xi\\
=\,&2u^4(q(0,t), t)+2u^2u_x^2(q(0,t), t)\\&+(u + u_x)(q(0,t), t) I_+(t)-(u - u_x)(q(0,t), t) I_-(t). 
\endaligned
\label{uxt}
\end{equation}
Above, we used the notation
\begin{align*}
I_+(t)=e^{q(0,t)} \int_{q(0,t)}^\infty e^{-\xi} \,y_t (\xi,t)\, d\xi,\qquad
I_-(t)=e^{-q(0,t)} \int_{-\infty}^{q(0,t)} e^\xi \,y_t (\xi, t)\, d\xi. 
\end{align*}

Using \eqref{bne} to substitute $y_t=-u^2y_x-buu_x\,y$ and integrating by parts, we obtain
\begin{equation*}
\aligned
I_+(t)=-e^{q(0,t)} \int_{q(0,t)}^\infty e^{-\xi} \,y(\xi,t)\left(u^2(\xi,t)+(b-2)uu_\xi(\xi,t)\right)\, d\xi,\\
I_-(t)=e^{-q(0,t)} \int_{-\infty}^{q(0,t)} e^\xi \,y(\xi, t)\left(u^2(\xi,t)-(b-2)uu_\xi(\xi,t)\right)\, d\xi. 
\endaligned
\end{equation*}
Following this, we substitute $y=u-u_{\xi\xi}$ in the integrands and we integrate by parts the terms containing a factor of $u_{\xi\xi}$. In this way, we derive
\begin{equation}\aligned
I_+(t)=&-u^2u_x(q(0,t), t)-\frac{b-2}{2}uu^2_x(q(0,t), t)\\
&-e^{q(0,t)} \int_{q(0,t)}^\infty e^{-\xi} \,\left(u^3+(b-3)u^2u_\xi+\frac{6-b}{2}uu_\xi^2+\frac{b-2}{2}u_\xi^3\right)(\xi,t)\,d\xi
\endaligned
\label{i+}
\end{equation}
and
\begin{equation}\aligned
I_-(t)=&-u^2u_x(q(0,t), t)+\frac{b-2}{2}uu^2_x(q(0,t), t)\\
&+e^{-q(0,t)} \int^{q(0,t)}_{-\infty} e^{\xi} \,\left(u^3-(b-3)u^2u_\xi+\frac{6-b}{2}uu_\xi^2-\frac{b-2}{2}u_\xi^3\right)(\xi,t)\,d\xi.
\endaligned
\label{i-}
\end{equation}


\subsection*{The derivation of the differential inequality}
First, from \eqref{yq0} and \eqref{uux} we deduce
\begin{equation}
(u+u_x)(x,t)\leq 0\ (\forall)\, x<q(0,t), \qquad (u-u_x)(x,t)\geq 0 \,(\forall)\, x>q(0,t).
\label{u+-}
\end{equation}
Next, note that
\begin{equation*}
\aligned
u^3+(b-3)u^2u_\xi&+\frac{6-b}{2}uu_\xi^2+\frac{b-2}{2}u_\xi^3\\
&=(u+u_\xi)\left(\frac{b-2}{2}u^2+(4-b)uu_\xi+\frac{b-2}{2}u_\xi^2\right)+\frac{4-b}{2}(u^3-3u^2u_\xi)
\endaligned
\end{equation*}
and
\begin{equation*}
\aligned
u^3-(b-3)u^2u_\xi&+\frac{6-b}{2}uu_\xi^2-\frac{b-2}{2}u_\xi^3\\
&=(u-u_\xi)\left(\frac{b-2}{2}u^2-(4-b)uu_\xi+\frac{b-2}{2}u_\xi^2\right)+\frac{4-b}{2}(u^3+3u^2u_\xi)
\endaligned
\end{equation*}
Moreover, 
\begin{equation*}
\frac{b-2}{2}z^2\pm(4-b)z+\frac{b-2}{2}\geq 0,\ (\forall)\,z\in\R,
\end{equation*}
is valid precisely when $b\geq 3$.

Combining all these facts in the context of \eqref{i+} and \eqref{i-}, we infer
\begin{equation*}\aligned
I_+(t)
\geq& -u^2u_x(q(0,t), t)-\frac{b-2}{2}uu^2_x(q(0,t), t)\\
&-\frac{4-b}{2}e^{q(0,t)} \int_{q(0,t)}^\infty e^{-\xi} \,\left(u^3-3u^2u_\xi\right)(\xi,t)\,d\xi\\
=& -u^2u_x(q(0,t), t)-\frac{b-2}{2}uu^2_x(q(0,t), t)-\frac{4-b}{2}u^3(q(0,t), t)
\endaligned
\end{equation*}
and
\begin{equation*}\aligned
I_-(t)
\geq& -u^2u_x(q(0,t), t)+\frac{b-2}{2}uu^2_x(q(0,t), t)\\
&+\frac{4-b}{2}e^{-q(0,t)} \int^{q(0,t)}_{-\infty} e^{\xi} \,\left(u^3+3u^2u_\xi\right)(\xi,t)\,d\xi\\
=& -u^2u_x(q(0,t), t)+\frac{b-2}{2}uu^2_x(q(0,t), t)+\frac{4-b}{2}u^3(q(0,t), t).
\endaligned
\end{equation*}
If we jointly use these estimates with \eqref{u+-} and \eqref{uxt}, we derive
\begin{equation}
\frac{d}{dt}\left\{2 u u_x(q(0,t), t)\right\}\leq (b-2)u^2(u^2-u_x^2)(q(0,t), t),
\label{dtu}
\end{equation}
which is the differential inequality that will drive the main norm inflation argument.

\begin{remark}
In the case of the NE ($b=3$), this estimate was first obtained by Jiang and Ni \cite{JN12}. It was later combined with the conservation of the $H^1$ norm by Yan, Li, and Zhang \cite{YLZ13} to derive a blow-up time result, which subsequently played a key role in the norm inflation argument of Guo et al. \cite{GLM19}.
\end{remark}


\section{Conclusion of the argument}
In this final section, we first reformulate our main result, Theorem \ref{main}, in a form better suited to the analytic tools developed earlier. We then introduce the initial data that will produce norm inflation and conclude by proving the reformulated theorem via a contradiction argument.


\subsection*{Reformulation step}
We first observe that if $\epsilon_2>\epsilon_1>0$ and the claims of Theorem \ref{main} hold for $\epsilon_1$, then they also hold for $\epsilon_2$. Indeed, the pair $(u_{\epsilon_1}, t_{\epsilon_1})$ corresponding to $\epsilon_1$ can be used for $\epsilon_2$. Therefore, it suffices to prove Theorem \ref{main} in the regime $0<\epsilon \ll 1$. 

Moreover, the statement can be relaxed by replacing $ \epsilon^{-1}$ in \eqref{-1} with $ \epsilon^{-\alpha}$ for any fixed $\alpha\in (0,1)$. Indeed, assume that for every $0<\epsilon \ll 1$ there exist $0<s_\epsilon<\epsilon$ and a solution $v_\epsilon$ of the $b$NE such that
\[
\|v_\epsilon(0)\|_{H^{3/2}}\lesssim \epsilon \qquad\text{and}\qquad \|v_\epsilon(s_\epsilon)\|_{H^{3/2}}\gtrsim \epsilon^{-\alpha}.
\]
Since $0<\epsilon \ll 1$ and $\alpha\in (0,1)$, we have $0<\epsilon^{1/\alpha}<\epsilon \ll 1$. Thus, if we let $u_\epsilon= v_{\epsilon^{1/\alpha}}$ and $t_\epsilon=s_{\epsilon^{1/\alpha}}$, then 
\[
0<t_\epsilon=s_{\epsilon^{1/\alpha}}<\epsilon^{1/\alpha}<\epsilon,
\]
\[
\|u_\epsilon(0)\|_{H^{3/2}}=\|v_{\epsilon^{1/\alpha}}(0)\|_{H^{3/2}}\lesssim \epsilon^{1/\alpha}<\epsilon,
\]
and
\[
\|u_\epsilon(t_\epsilon)\|_{H^{3/2}}= \|v_{\epsilon^{1/\alpha}}(s_{\epsilon^{1/\alpha}})\|_{H^{3/2}}\gtrsim (\epsilon^{1/\alpha})^{-\alpha}= \epsilon^{-1}.
\]

Hence, Theorem \ref{main} is equivalent to the following formulation.

\begin{theorem}
Let $b\geq 3$ and $0<\alpha<1$. For every $0<\epsilon\ll1$, there exist a solution $u_\epsilon$ to the $b$NE  and $0<t_\epsilon<\epsilon$ such that 
\begin{equation}
\|u_\epsilon(0)\|_{H^{3/2}}\lesssim \epsilon \qquad\text{and}\qquad\|u_\epsilon(t_\epsilon)\|_{H^{3/2}}\gtrsim \epsilon^{-\alpha}.\label{0a}
\end{equation}
\label{main-2}
\end{theorem} 


\subsection*{Selection of the initial data}
As in the Novikov equation case ($b=3$), we use the sequence of data introduced by Guo et al. in Remark 3.3 of \cite{GLM19}. Taking $p=r=2$ in that remark, consider the sequence $(u_{0,K})_{K\geq 1}$ given by
\begin{equation*}
u_{0,K}(x)=\sum_{k=1}^K \frac{1}{k^{2/3}}(1-\partial^2_x)^{-1}(\phi_k)(x),\quad \forall\,K\geq 1,
\end{equation*}
where
\begin{equation*}
\phi_k(x)=2^k\phi(2^k x), \quad \phi(x)=\psi(x+2)-\psi(x-200),
\end{equation*}
and $\psi:\mathbb{R}\to [0,1]$ is even, smooth, radially decreasing, supported in $[-8/5,8/5]$, and equal to $1$ on $[-5/4, 5/4]$. 

Most properties of $u_{0,K}$ were established in \cite{GLM19}; the remaining ones follow by straightforward adaptations of their arguments. The sequence satisfies:

\begin{itemize}

\item for each $K\geq1$, $u_{0, K}$ is real-valued and $u_{0,K}\in H^\infty$;

\item for each $K\geq1$, $y_{0,K}=(1-\partial^2_x)u_{0,K}$ satisfies
\begin{equation}\label{y0k}
y_{0, K}\geq 0\ \text{on}\ (-\infty, 0], \quad y_{0, K}(0)=0, \quad y_{0, K}\leq 0 \ \text{on}\ [0,\infty);
\end{equation}

\item the Sobolev estimate
\begin{equation*}
\|u_{0, K}\|_{H^s}\lesssim
\begin{cases}
1 & \text{if } s \leq \tfrac{3}{2}, \\
\dfrac{2^{K(s - 3/2)}}{K^{2/3}} & \text{if } s > \tfrac{3}{2}.
\end{cases}
\end{equation*}
holds uniformly in $K$ (i.e., the constant consolidated in $\lesssim$ is independent of $K$);

\item moreover,
\begin{equation} u_{0,K}(0)\gtrsim 1\qquad \text{and}\qquad
u_{0, K}'(0)\sim -K^{1/3}\label{u0k}
\end{equation}
hold uniformly in $K$.
\end{itemize}

\subsection*{Proof of Theorem \ref{main-2}}
We argue by contradiction. Assume that there exists $0<\epsilon\ll1$ such that for every solution $u_\epsilon$ to the $b$NE with 
\[
\|u_\epsilon(0)\|_{H^{3/2}}\lesssim \epsilon
\] 
and maximal time of existence $T^{\text{max}}_\epsilon>\epsilon$, one has
\begin{equation*}
\sup_{t\in [0,\epsilon]} \|u_\epsilon(t)\|_{H^{3/2}}\ll \epsilon^{-\alpha}.
\end{equation*}
We construct a solution violating this property by producing finite-time blowup before $\epsilon$.

For the given $\epsilon$, let $(v_{0,K})_{K\geq 1}$ be defined as
\begin{equation*}
v_{0, K}(x)=\epsilon \,u_{0,K}(x).
\end{equation*}
Using the properties of $(u_{0,K})_{K\geq 1}$ and Theorem \ref{hh}, we infer that the Cauchy problem \eqref{cp} with data $v_{0, K}\in H^\infty\subset H^s$ admits a unique saturated solution 
\[
v_K\in C([0, T^{\text{max}}_K); H^s)
\] 
with 
\[
T^{\text{max}}_K\gtrsim \frac{2s-3}{\|v_{0,K}\|^2_{H^s}}=  \frac{2s-3}{\epsilon^2\,\|u_{0,K}\|^2_{H^s}}\gtrsim \frac{K^{4/3}(2s-3)}{\epsilon^2 \,2^{K(2s-3)}},
\]
and
\[
\|v_K(t)\|_{H^s}\leq 2 \|v_{0,K}\|_{H^s}\lesssim \frac{\epsilon\, 2^{K(s-3/2)}}{K^{2/3}}, \qquad \forall\, 0\leq t\lesssim \frac{K^{4/3}(2s-3)}{\epsilon^2 \,2^{K(2s-3)}}.
\]
Moreover,
\begin{equation*}
\|v_{0,K}\|_{H^{3/2}}= \epsilon\,\|u_{0,K}\|_{H^{3/2}}\lesssim \epsilon.
\end{equation*}

It follows that, if 
\begin{equation}\boxed{
\frac{K^{4/3}(2s-3)}{\epsilon^2 \,2^{K(2s-3)}}\gg \epsilon,}\label{ke}
\end{equation}
then
\begin{equation}
T^{\text{max}}_K>\epsilon\label{tke}
\end{equation}
and $v_K$ satisfies the hypothesis of the contradiction assumption. In particular, we have
\begin{equation}
\sup_{t\in [0,\epsilon]} \|v_K(t)\|_{H^{3/2}}\ll \epsilon^{-\alpha}.\label{32}
\end{equation}

Next, based on our discussion in Section \ref{3} and taking into account the properties of $v_{0,K}$ (e.g., $v_{0, K}\in H^\infty$), the particle line associated to the solution $v_K$ is well-defined, with
\[
q=q(x,t)\in C^1(\R\times [0,T^{\text{max}}_K); \R).
\]
If we let
\[
f(t)=(v_K v_{K,x})(q(0,t), t),
\]
then, by \eqref{dtu}, we deduce
\begin{equation*}
\frac{d}{dt}f(t)\leq -\frac{b-2}{2}f^2(t) +\frac{b-2}{2}v_K^4(q(0,t), t),\quad \forall\,t\in [0,T_K^{\text{max}}).
\end{equation*}
Using \eqref{tke}, \eqref{32}, and the embedding $H^{3/2}\subset L^\infty$, we derive
\begin{equation}
\frac{d}{dt}f(t)\leq -\frac{b-2}{2}f^2(t) +c_1\,\epsilon^{-4\alpha},\quad \forall\,t\in [0,\epsilon],\label{ft}
\end{equation}
for some $0<c_1\ll 1$. 

At $t=0$, we have
\[
f(0)=(v_K v_{K,x})(0, 0)=v_{0, K}(0)v'_{0, K}(0)=\epsilon^2 u_{0, K}(0)u'_{0, K}(0),
\]
which, by \eqref{u0k}, implies
\begin{equation}
f(0)\lesssim -\epsilon^2\, K^{1/3}.\label{f0ek}
\end{equation}
To initiate the blowup, we want to ensure that 
\begin{equation}
\frac{d}{dt}f(0)<0.\label{df0}
\end{equation}
By \eqref{ft}, we have
\[
\frac{d}{dt}f(0)\leq -\frac{b-2}{2}f^2(0) +c_1\,\epsilon^{-4\alpha},
\]
and, thus, \eqref{df0} holds if we enforce
\begin{equation}
f(0)< -c_2\,\epsilon^{-2\alpha}, \label{f0}
\end{equation}
for some $c_2>0$ that we take to satisfy 
\begin{equation}
c_1< \frac{b-2}{2}\,c_2^2.\label{c}
\end{equation}
Using \eqref{f0ek}, we see that \eqref{f0} holds if 
\begin{equation}
\boxed{\epsilon^2\, K^{1/3}\gtrsim \epsilon^{-2\alpha}.}\label{f0e}
\end{equation}

Following this, we prove:

\begin{lemma}
Assume that $f=f(t)$ satisfies \eqref{ft}, \eqref{df0}, and \eqref{f0} (with \eqref{c} as part of it). Then
\begin{equation}
f(t)< -c_2\,\epsilon^{-2\alpha}, \quad \forall\,t\in [0,\epsilon].\label{fce}
\end{equation}
\end{lemma}
\begin{proof}
Assume by contradiction that for some $t_0\in(0,\epsilon]$ we have
\[
f(t)< -c_2\,\epsilon^{-2\alpha}, \quad \forall\,t\in [0,t_0),
\] 
and 
\[
f(t_0)= -c_2\,\epsilon^{-2\alpha}.
\]
Then we infer
\[
\frac{d}{dt}f(t)\leq -\frac{b-2}{2}f^2(t) +c_1\,\epsilon^{-4\alpha}<(c_1-\frac{b-2}{2}\,c_2^2)\epsilon^{-4\alpha}<0, \quad \forall\,t\in [0,t_0),
\]
which implies
\[
f(t_0)=\lim_{t\to t_0-}f(t)\leq f(0)< -c_2\,\epsilon^{-2\alpha}.
\]
This provides the desired contradiction.
\end{proof}

Now, by using \eqref{ft} and \eqref{fce}, we obtain
\begin{equation*}
\aligned
\frac{d}{dt}f(t)\leq -\frac{b-2}{2}f^2(t) &+c_1\,\epsilon^{-4\alpha}= -\frac{b-2}{2}f^2(t)+\frac{c_1}{\frac{b-2}{2}\,c_2^2} \,\frac{b-2}{2}\,c_2^2\,\epsilon^{-4\alpha}\\
&<\left(-1+\frac{c_1}{\frac{b-2}{2}\,c_2^2}\right) f^2(t), \quad \forall\,t\in [0,\epsilon],
\endaligned
\end{equation*}
which can be restated as
\[
\frac{d}{dt}\left\{\frac{1}{f(t)}\right\}>1-\frac{c_1}{\frac{b-2}{2}\,c_2^2}, \quad \forall\,t\in [0,\epsilon].
\]
Since $f(t)<0$ for all $t\in [0,\epsilon]$, we derive by integration that
\[
- \frac{1}{f(0)}>\frac{1}{f(t)}- \frac{1}{f(0)}>\left(1-\frac{c_1}{\frac{b-2}{2}\,c_2^2}\right)t, \quad \forall\,t\in [0,\epsilon].
\]
Based on \eqref{c}, we deduce
\[
1-\frac{c_1}{\frac{b-2}{2}\,c_2^2}>0
\]
and blowup occurs before $\epsilon$ if 
\[
- \frac{1}{f(0)}\leq \left(1-\frac{c_1}{\frac{b-2}{2}\,c_2^2}\right)\epsilon,
\]
which holds if 
\[
\epsilon^2\,f^2(0)\gtrsim 1.
\]
By \eqref{f0ek} and \eqref{f0e}, 
\[
\epsilon^2\,f^2(0)\gtrsim \epsilon^{2-4\alpha},
\]
and, since $0<\epsilon\ll 1$, it suffices to impose 
\begin{equation}
\boxed{\alpha\geq 1/2.}\label{al}
\end{equation}

Finally, given $0<\epsilon\ll 1$ and $0<\alpha<1$, we consolidate all the conditions we need to ensure in order for the argument to go through: \eqref{ke}, \eqref{f0e}, and \eqref{al}. First, for \eqref{al}, we let 
\[
\alpha= 1/2.
\]
Then \eqref{f0e} becomes
\[
\epsilon^3\, K^{1/3}\gtrsim 1,
\]
which is valid if we take
\[
K=\epsilon^{-9}.
\]
Then \eqref{ke} becomes
\[
\frac{\epsilon^{-15}(2s-3)}{2^{\epsilon^{-9}(2s-3)}}\gg 1.
\]
By choosing 
\[
s=\frac{3+\epsilon^9}{2},
\]
the previous estimate can be restated as
\[
\epsilon^{-6}\gg 1,
\]
which is valid since $0<\epsilon\ll 1$.

With these choices for $\alpha$, $K$, and $s$, the constructed solution $v_K$ to the $b$NE has maximal time of existence $T^{\text{max}}_K>\epsilon$, while the function 
\[
t\mapsto (v_K v_{K,x})(q(0,t), t)
\]
blows up before time $\epsilon$. This contradicts the well-definedness of the particle trajectory on the entire interval of existence. Hence, Theorem \ref{main-2} follows.


\section*{Acknowledgements}
Dan-Andrei Geba is deeply grateful to God and the Holy Theotokos for Their love, patience, and mercy, which inspired him while working on this paper. He is also thankful to the co-authors for inviting him to collaborate on this project and to Alex Himonas for the opportunity to visit and work with him at Notre Dame. 

Finally, the second author was partially supported by a grant from the 
Simons Foundation (\#524469 to Alex Himonas).  


\bibliographystyle{amsplain}
\bibliography{bNE}

@book {L13,
    AUTHOR = {Lannes, D.},
     TITLE = {The water waves problem},
    SERIES = {Mathematical Surveys and Monographs},
    VOLUME = {188},
      NOTE = {Mathematical analysis and asymptotics},
 PUBLISHER = {American Mathematical Society, Providence, RI},
      YEAR = {2013},
     PAGES = {xx+321},}

@article {CH93,
    AUTHOR = {Camassa, R. and Holm, D.},
     TITLE = {An integrable shallow water equation with peaked solitons},
   JOURNAL = {Phys. Rev. Lett.},
  FJOURNAL = {Physical Review Letters},
    VOLUME = {71},
      YEAR = {1993},
    NUMBER = {11},
     PAGES = {1661--1664},
}

@article {FF81,
    AUTHOR = {Fuchssteiner, B. and Fokas, A. S.},
     TITLE = {Symplectic structures, their {B}\"acklund transformations and
              hereditary symmetries},
   JOURNAL = {Phys. D},
  FJOURNAL = {Physica D. Nonlinear Phenomena},
    VOLUME = {4},
      YEAR = {1981/82},
    NUMBER = {1},
     PAGES = {47--66},}

@article {HS03,
    AUTHOR = {Holm, D. and Staley, M.},
     TITLE = {Wave structure and nonlinear balances in a family of
              evolutionary {PDE}s},
   JOURNAL = {SIAM J. Appl. Dyn. Syst.},
  FJOURNAL = {SIAM Journal on Applied Dynamical Systems},
    VOLUME = {2},
      YEAR = {2003},
    NUMBER = {3},
     PAGES = {323--380},
}

@article {HS032,
    AUTHOR = {Holm, D. and Staley, M.},
     TITLE = {Nonlinear balance and exchange of stability of dynamics of
              solitons, peakons, ramps/cliffs and leftons in a {$1+1$}
              nonlinear evolutionary {PDE}},
   JOURNAL = {Phys. Lett. A},
  FJOURNAL = {Physics Letters. A},
    VOLUME = {308},
      YEAR = {2003},
    NUMBER = {5-6},
     PAGES = {437--444},
}

@article {R01,
    AUTHOR = {Rodr\'iguez-Blanco, G.},
     TITLE = {On the {C}auchy problem for the {C}amassa-{H}olm equation},
   JOURNAL = {Nonlinear Anal.},
  FJOURNAL = {Nonlinear Analysis. Theory, Methods \& Applications. An
              International Multidisciplinary Journal},
    VOLUME = {46},
      YEAR = {2001},
    NUMBER = {3},
     PAGES = {309--327},
}

@article {LO00,
    AUTHOR = {Li, Y. and Olver, P.},
     TITLE = {Well-posedness and blow-up solutions for an integrable
              nonlinearly dispersive model wave equation},
   JOURNAL = {J. Differential Equations},
  FJOURNAL = {Journal of Differential Equations},
    VOLUME = {162},
      YEAR = {2000},
    NUMBER = {1},
     PAGES = {27--63},}

@article {GL11,
    AUTHOR = {Gui, G. and Liu, Y.},
     TITLE = {On the {C}auchy problem for the {D}egasperis-{P}rocesi
              equation},
   JOURNAL = {Quart. Appl. Math.},
  FJOURNAL = {Quarterly of Applied Mathematics},
    VOLUME = {69},
      YEAR = {2011},
    NUMBER = {3},
     PAGES = {445--464},}

@article {GLT08,
    AUTHOR = {Gui, G. and Liu, Y. and Tian, L.},
     TITLE = {Global existence and blow-up phenomena for the peakon
              {$b$}-family of equations},
   JOURNAL = {Indiana Univ. Math. J.},
  FJOURNAL = {Indiana University Mathematics Journal},
    VOLUME = {57},
      YEAR = {2008},
    NUMBER = {3},
     PAGES = {1209--1234},
}

@article {HH22,
    AUTHOR = {Himonas, A. and Holliman, C.},
     TITLE = {Instability and nonuniqueness for the {$b$}-{N}ovikov
              equation},
   JOURNAL = {J. Nonlinear Sci.},
  FJOURNAL = {Journal of Nonlinear Science},
    VOLUME = {32},
      YEAR = {2022},
    NUMBER = {4},
     PAGES = {Paper No. 46, 29},
}

@article {HHK18,
    AUTHOR = {Himonas, A. and Holliman, C. and Kenig, C.},
     TITLE = {Construction of 2-peakon solutions and ill-posedness for the
              {N}ovikov equation},
   JOURNAL = {SIAM J. Math. Anal.},
  FJOURNAL = {SIAM Journal on Mathematical Analysis},
    VOLUME = {50},
      YEAR = {2018},
    NUMBER = {3},
     PAGES = {2968--3006},
}

@article {HGH16,
    AUTHOR = {Himonas, A. and Grayshan, K. and Holliman,
              C.},
     TITLE = {Ill-posedness for the {$b$}-family of equations},
   JOURNAL = {J. Nonlinear Sci.},
  FJOURNAL = {Journal of Nonlinear Science},
    VOLUME = {26},
      YEAR = {2016},
    NUMBER = {5},
     PAGES = {1175--1190},
}

@article {HHG14,
    AUTHOR = {Himonas, A. and Holliman, C. and Grayshan,
              K.},
     TITLE = {Norm inflation and ill-posedness for the
              {D}egasperis-{P}rocesi equation},
   JOURNAL = {Comm. Partial Differential Equations},
  FJOURNAL = {Communications in Partial Differential Equations},
    VOLUME = {39},
      YEAR = {2014},
    NUMBER = {12},
     PAGES = {2198--2215},
}

@article {HH14,
    AUTHOR = {Himonas, A. and Holliman, C.},
     TITLE = {The {C}auchy problem for a generalized {C}amassa-{H}olm
              equation},
   JOURNAL = {Adv. Differential Equations},
  FJOURNAL = {Advances in Differential Equations},
    VOLUME = {19},
      YEAR = {2014},
    NUMBER = {1-2},
     PAGES = {161--200},
}

@article {GLM19,
    AUTHOR = {Guo, Z. and Liu, X. and Molinet, L. and Yin,
              Z.},
     TITLE = {Ill-posedness of the {C}amassa-{H}olm and related equations in
              the critical space},
   JOURNAL = {J. Differential Equations},
  FJOURNAL = {Journal of Differential Equations},
    VOLUME = {266},
      YEAR = {2019},
    NUMBER = {2-3},
     PAGES = {1698--1707},
}

@article {B06,
    AUTHOR = {Byers, P.},
     TITLE = {Existence time for the {C}amassa-{H}olm equation and the
              critical {S}obolev index},
   JOURNAL = {Indiana Univ. Math. J.},
  FJOURNAL = {Indiana University Mathematics Journal},
    VOLUME = {55},
      YEAR = {2006},
    NUMBER = {3},
     PAGES = {941--954},
}

@article {JN12,
    AUTHOR = {Jiang, Z. and Ni, L.},
     TITLE = {Blow-up phenomenon for the integrable {N}ovikov equation},
   JOURNAL = {J. Math. Anal. Appl.},
  FJOURNAL = {Journal of Mathematical Analysis and Applications},
    VOLUME = {385},
      YEAR = {2012},
    NUMBER = {1},
     PAGES = {551--558},}

@article {YLZ13,
    AUTHOR = {Yan, W. and Li, Y. and Zhang, Y.},
     TITLE = {The {C}auchy problem for the {N}ovikov equation},
   JOURNAL = {NoDEA Nonlinear Differential Equations Appl.},
  FJOURNAL = {NoDEA. Nonlinear Differential Equations and Applications},
    VOLUME = {20},
      YEAR = {2013},
    NUMBER = {3},
     PAGES = {1157--1169},}

@article {HH12,
    AUTHOR = {Himonas, A. and Holliman, C.},
     TITLE = {The {C}auchy problem for the {N}ovikov equation},
   JOURNAL = {Nonlinearity},
  FJOURNAL = {Nonlinearity},
    VOLUME = {25},
      YEAR = {2012},
    NUMBER = {2},
     PAGES = {449--479},
}

@article {N09,
    AUTHOR = {Novikov, V.},
     TITLE = {Generalizations of the {C}amassa-{H}olm equation},
   JOURNAL = {J. Phys. A},
  FJOURNAL = {Journal of Physics. A. Mathematical and Theoretical},
    VOLUME = {42},
      YEAR = {2009},
    NUMBER = {34},
     PAGES = {342002, 14},
}

@article {HK09,
    AUTHOR = {Himonas, A. and Kenig, C.},
     TITLE = {Non-uniform dependence on initial data for the {CH} equation
              on the line},
   JOURNAL = {Differential Integral Equations},
  FJOURNAL = {Differential and Integral Equations. An International Journal
              for Theory \& Applications},
    VOLUME = {22},
      YEAR = {2009},
    NUMBER = {3-4},
     PAGES = {201--224},
}

@book {W74,
    AUTHOR = {Whitham, G. B.},
     TITLE = {Linear and nonlinear waves},
    SERIES = {Pure and Applied Mathematics},
 PUBLISHER = {Wiley-Interscience [John Wiley \& Sons], New
              York-London-Sydney},
      YEAR = {1974},
     PAGES = {xvi+636},}

\end{document}